\documentclass[a4paper,11pt]{article}

\usepackage[blocks]{authblk}
\usepackage{amsthm,latexsym,amsmath,amssymb,amsfonts,euscript,dsfont}
\usepackage{epsfig}
\usepackage{mathtools}
\usepackage{mathrsfs}
\usepackage{relsize} 
\usepackage{scalerel}
\usepackage{comment}
\usepackage{ushort}
\usepackage{color}
\usepackage{url}
\usepackage{hyperref}

\usepackage{ulem} 

\newtheorem{theorem}{{\bf{\small T}{\scriptsize HEOREM}}}[section]
\newtheorem*{maintheorem}{{\bf{\small M}{\scriptsize AIN THEOREM}}}
\newtheorem{corollary}{{\bf{\small C}{\scriptsize OROLLARY}}}[section]
\newtheorem{proposition}{{\bf{\small P}{\scriptsize ROPOSITION}}}[section]
\newtheorem{lemma}{{\bf{\small L}{\scriptsize EMMA}}}[section]
\newtheorem{remark}{{\bf{\small R}{\scriptsize EMARK}}}[section]

\newtheorem{definition}{{\bf{\small D}{\scriptsize EFINITION}}}[section]

\renewenvironment{proof}[1]
{\noindent{{\bf{\small{ P}{\scriptsize ROOF}}}.}\hspace{0.1cm} #1} {$\;\qed$\newline}

\newenvironment{sketchofproof}[1]
{\noindent{{\bf{\small{ S}{\scriptsize KETCH OF THE PROOF}}}.}\hspace{0.1cm} #1} {$\;\qed$\newline}

\def\R{\mathds R}
\def\N{\mathds N}
\def\Z{\mathds Z}
\newcommand{\Zd}{\mathds Z^d}
\def\un{\mathds{1}}
\def\E{\mathds E}

\def\bgamma{\boldsymbol\gamma}

\newcommand{\dd}{\mathop{}\!\mathrm{d}}

\newcommand{\tv}{{\scriptscriptstyle \mathrm{TV}}}

\DeclareMathOperator{\e}{\mathrm{e}}

\font\gfont=cmmi8 scaled \magstep{1.2}
{2}

\newcommand{\gdelta}{\hbox{\gfont \char14}}

\newcommand{\db}{\bar{\mathrm{d}}}

\newcommand{\geta}{\hbox{\gfont \char17}}

\newcommand{\gmu}{\hbox{\gfont \char22}}

\newcommand{\gsigma}{\hbox{\gfont \char27}}

\newcommand{\gcb}[1]{\mathrm{GCB}\!\left(#1\right)}

\begin{document}

\title{Gaussian concentration and uniqueness of equilibrium states in lattice systems}

\author[1]{J.-R. Chazottes
\thanks{Email: \texttt{jeanrene@cpht.polytechnique.fr}}}

\author[1,3]{J. Moles
}

\author[2]{F. Redig
}

\author[3]{E. Ugalde
}

\affil[1]{Centre de Physique Th\'eorique, CNRS, Institut Polytechnique de Paris, France}
\affil[2]{Institute of Applied Mathematics, Delft University of Technology, 
The Netherlands}
\affil[3]{Instituto de F\'{\i}sica, Universidad Aut\'onoma de San Luis Potos\'{\i}, 
M\'exico}

\date{Dated: \today}

\maketitle

\begin{abstract}
We consider equilibrium states (that is, shift-invariant Gibbs measures) on the configuration space $S^{\Zd}$ where $d\geq 1$ and $S$ is a finite
set.
We prove that if an equilibrium state for a shift-invariant uniformly summable potential satisfies a Gaussian concentration bound,
then it is unique. Equivalently, if there exist several equilibrium states for a potential, none of them can satisfy such a bound.

\smallskip

\end{abstract}

\maketitle


\tableofcontents

\section{Introduction and main result}

The phenomenon we are interested in, which goes under the name of ``concentration inequalities'', is that if a function of
many ``weakly dependent'' random variables does not depend too much on any of them, then it is concentrated around its
expected value.
A key feature of this phenomenon is that it is \textit{non-asymptotic}, in contrast with the usual \textit{limit} theorems
where the number of random variables has to tend to infinity. Recall that the three main types of classical limit theorems are
the law of large numbers, the central limit theorem, and large deviations.
Another key feature of concentration inequalities is that they allow to deal with functions of random variables defined in an arbitrary
way, provided they are ``smooth enough'', in contrast with classical limit theorems which deal with sums of random variables.
Concentration inequalities made a paradigm shift in probability and statistics, but also in discrete mathematics, in geometry and in
functional analysis, see e.g. \cite{blm,dp,ledoux,Wbook}.

In this paper, we consider Gibbs measures on the configuration space $\Omega=S^{\Zd}$ where $S$ is a finite set
and $d\geq 1$. Postponing precise definitions till next section, a probability measure $\mu$ satisfies a Gaussian concentration bound
if there exists a constant $D>0$ such that, for any local function $F:\Omega\to\R$,
\begin{equation}\label{gcbineq-intro}
\int \e^{F-\int F\dd\mu} \dd\mu\leq \e^{D\sum_{x\in \Lambda_{\scaleto{F}{3pt}}} \delta_x(F)^2}
\end{equation}
where $\Lambda_F$ is the (finite) set of sites $x\in\Zd$ such that $\delta_x(F)\neq 0$, and $\delta_x(F)$ is the largest value of $|F(\omega)-F(\omega')|$
taken over the configurations $\omega$ and $\omega'$ differing only at site $x$. Note that $D$ does not depend on $F$, and in particular not on
$\Lambda_{\scaleto{F}{4pt}}$.
By a standard argument (recalled later on), \eqref{gcbineq-intro}
implies a control on the fluctuations of $F$ around $\int F\dd\mu$: for all $u>0$, we have
\begin{equation}\label{intro-dev-ineq}
\mu\Big\{\omega\in \Omega : F(\omega)\geq  \int F\dd\mu+u \Big\} \leq
\exp\left(-\frac{u^2}{4D\sum_{x\in \Lambda_{\scaleto{F}{3pt}}} \delta_x(F)^2}\right).
\end{equation}
For instance, if $S=\{-1,+1\}$, $F$ can be $\sum_{x\in\Lambda_n} \omega_x/|\Lambda_n|$, which is the magnetization inside the ``cube''
$\Lambda_n=[-n,n]^d\cap \Zd$. One can check immediately that the previous bound is exponentially small in the volume $|\Lambda_n|=(2n+1)^d$.
But we can consider much more general $F$'s, and in particular nonlinear or implicitly defined functions.
The bound \eqref{gcbineq-intro} was first proved in \cite{Kul} for potentials satisfying Dobrushin's uniqueness condition, with a
constant $D$ explicitly related to Dobrushin's interdependence matrix. This covers, for instance, finite-range potentials at sufficiently high temperature.
Not surprisingly, one cannot expect that a Gaussian concentration holds for the (ferromagnetic) Ising model at temperatures below
the critical one, because of the surface-order large deviations of the magnetization (see \cite{CCKR} for more details).
In \cite{CCKR}, the authors proved that a ``stretched-exponential'' concentration bound holds for the ``$+$'' phase and the ``$-$'' phase
of this model at sufficiently low temperature.
Concerning the various applications of these concentration bounds, we will not describe them here and refer the reader to \cite{Kul,CCR}.

What happens for the Ising model raises the following general question:
\begin{quote}
Suppose that a potential admits several Gibbs measures. Is it true that none of these measures satisfy a Gaussian concentration bound?
\end{quote}
Equivalently, we ask:
\begin{quote}
If a Gibbs measure for a given potential satisfies a Gaussian concentration bound, is it unique?
\end{quote}
In this paper, we prove the following result which answers this question in the case of \textit{shift-invariant} Gibbs measures, that is,
\textit{equilibrium states}.\footnote{In the class of potentials we consider, shift-invariant Gibbs measures coincide with equilibrium states \cite[Theorem 4.2]{ruelle2004}.}
\begin{maintheorem}
If an ergodic equilibrium state for a shift-invariant absolutely summable potential satisfies a Gaussian concentration
bound, then it must be the unique equilibrium state for this potential.
\end{maintheorem}
Our theorem can be paraphrased by saying that nonuniqueness of the equilibrium states for a potential prevents Gaussian concentration.
Observe that the statement is about \textit{ergodic} equilibrium states because, as we prove below, if a shift-invariant probability satisfies a
Gaussian concentration bound, then it must be ergodic (Proposition \ref{GCBimpliesergodicity}).

We will prove the above theorem in two different ways. \newline
The first way is based on ideas which were put forward in ergodic theory to study the existence of finitary codings from a finite-valued i.i.d. process to certain
ergodic processes. Two central notions turn out to be the ``blowing-up property'' and the ``positive (lower) relative entropy property''.
Without going into detail, it was proved in \cite{martonshields} that if a process is finitely determined then it has the blowing-up property, which in turn implies
the positive relative entropy property. Here, we use the fact that the Gaussian concentration bound implies the blowing-up property, and then use part of the
variational principle which says that the relative entropy of two distinct equilibrium states for the same potential is equal to zero. Hence the blowing-up
property cannot hold, therefore it is not possible to have a Gaussian concentration bound. In fact, we establish an abstract result (Theorem
\ref{abstracttheorem}) which states that if a probability measure satisfies a Gaussian concentration bound, then it has the positive lower relative entropy property.
Technically speaking, we follow some methods that can be found in \cite{martonshields}. The passage from $d=1$ to $d\geq 2$ (that is, going from processes to
random measures) poses no  difficulty. Although these methods are known to specialists of ergodic theory, they are probably not as well-known in the mathematical physics
literature.
For this  reason and to make this article self-contained, we give a detailed and complete presentation of the methods. Let us also mention that the line of reasoning
we follow here was used for the first time in \cite{CGT} in the context of chains of unbounded memory, a.k.a., chains with complete connections.
Using the methods of \cite{martonshields}, it was proved in \cite{vdberg-steif} that, for Markov random fields (including the nearest neighbour Ising model),
a phase transition is an obstruction for the finitely determined property. A natural question is thus: How does the Gaussian concentration bound relate
to the finitely determined property? A reasonable conjecture is that the former implies the latter, but we have not investigated this question.\newline
The other way of proving the main theorem is via large deviations. Let us sketch the proof. Let us first recall that, for equilibrium states, one can prove a ``large
deviation principle'' for the empirical measure with a ``rate function'' which is the relative entropy. Roughly speaking, the content of such a result is that, if we take an equilibrium state $\mu$, and an ergodic probability measure $\nu$ different from $\mu$, then
\begin{equation}\label{ld-intro}
\text{Prob}_{\mu}\big( \omega_{\Lambda_n}\;\text{is typical for}\;\nu\big)\asymp \e^{-|\Lambda_n| h(\nu|\mu)}
\end{equation}
where $h(\nu|\mu)$ is the relative entropy of $\nu$ with respect to $\mu$, which will be precisely defined below. The symbol ``$\asymp$'' means asymptotic
equivalence on the logarithmic scale in the following sense: $a_{\Lambda_n}\asymp b_{\Lambda_n}$ means
$\lim_{n\to+\infty}|\Lambda_n|^{-1}\log a_{\Lambda_n} = \lim_{n\to+\infty}|\Lambda_n|^{-1}\log b_{\Lambda_n}$.
What ``typical'' means is in relation to the ergodic theorem: If $\omega_{\Lambda_n}$ is typical for $\nu$, this means that the average
$|\Lambda_n|^{-1}\sum_{x\in\Lambda_n} f(\theta_x \omega)$
of every continuous function $f:\Omega\to\R$ under the shift (translation) over $\Lambda_n$ converges to $\int f\dd\nu$, as $n\to+\infty$.
The set of typical configurations for $\nu$ is of measure $0$ for $\mu$, and thus what \eqref{ld-intro} roughly says is that the probability that a
configuration $\omega$, taken from the probability distribution $\mu$, looks in $\Lambda_n$ like a typical configuration from $\nu$ decays
exponentially in the volume $|\Lambda_n|=(2n+1)^d$.
Coming back to our main result, the idea behind its proof is then simple: If $\mu$ and $\nu$ are two distinct (ergodic) equilibrium states for the same potential, then $h(\nu|\mu)=0$, whence the probability in \eqref{ld-intro} decays sub-exponentially in $\Lambda_n$.
Since $\mu\neq \nu$, there exists a local function $f$ such that $\int f\dd\nu=\int f\dd\mu+\varepsilon$ for some $\varepsilon>0$.

Now, suppose that $\mu$ satisfies a Gaussian concentration bound. Then $|\Lambda_n|^{-1}\sum_{x\in\Lambda_n} f\circ \theta_x$ must sharply
concentrate around $\int f\dd\mu$, in the sense that the $\mu$-probability that this average is larger than $\int f\dd\mu+\varepsilon$ (call this event
$E_n$), decays exponentially fast in $|\Lambda_n|$. This is a consequence of \eqref{intro-dev-ineq}.
But, at the same time, by the large deviation principle, the $\mu$-probability that this average
is concentrated around $\int f\dd\nu$ (call this event $E'_n$) decays only sub-exponentially fast. Since $E'_n\subset E_n$, we get a contradiction
since $\mu(E'_n)\leq \mu(E_n)$, which is incompatible with the fact that $\mu(E'_n)$ is sub-exponentially small in $|\Lambda_n|$, whereas
$\mu(E_n)$ is exponentially small in $|\Lambda_n|$.

Let us comment on the two proofs of the main theorem. The one based on large deviations is short and simple, given that we have a large deviation principle at our
disposal. Moreover, it works for more general spin spaces $S$, in particular non-discrete spaces.
The other proof is longer, but it puts forward the blowing-up property which is a remarkable property.
It also connects Gaussian concentration with the positivity of the (lower) relative entropy (Theorem \ref{abstracttheorem}), which is of independent interest (and not tied to equilibrium states).

Finally, let us mention that, strictly speaking, we do not consider hard-core lattice gas models. But our result extends to that situation by ad-hoc
modifications. Notice that the first proof works as well, whereas the proof based on large deviations also works, provided we invoke a large deviation
principle proved in \cite{EKW}.

\section{Setting}

We set some basic notation.
The configuration space is $\Omega=S^{\Zd}$, where $S$ is a finite set, and $d$ an integer greater than or equal to $1$.
We endow $\Omega$ with the product topology that is generated by cylinder sets, which makes it a compact metrizable space.
We denote by $\mathfrak{B}$ the Borel $\sigma$-algebra which coincides with the $\sigma$-algebra generated by cylinder sets.

An element $x$ of $\Zd$ (hereby called a site) can be written as a vector $(x_1,\ldots,x_d)$ in the canonical base of the lattice $
\Zd$.
If $\Lambda$ is a finite subset of $\Zd$, denote by $|\Lambda|$ its cardinality.
If $\Lambda$ is a finite subset of $\Zd$, we will write $\Lambda \Subset \Zd$.

For $\Lambda\subset\Zd$, $\gsigma,\geta\in \Omega$, we denote by $\gsigma_\Lambda\geta_{\Lambda^{\!c}}$ the configuration which
agrees with $\gsigma$ on $\Lambda$ and with $\geta$ on $\Lambda^{\!c}$. We denote by $\mathfrak{B}_\Lambda$ the $\sigma$-algebra
generated by the coordinate maps $\omega\mapsto \omega_x$ when $x$ is restricted to $\Lambda$.
We need to define centered ``cubes'': for every $n\in\Z_+$, let
\[
 \Lambda_n=\big\{x\in\Zd : -n\leq x_i\leq n,\,i=1,2,\ldots,d\big\}.
\]
Given $\Lambda\Subset\Zd$, an element $p_\Lambda$ of $S^{\Lambda}$ is called a pattern with shape $\Lambda$, or simply a
pattern. We will write $p_n$ for a pattern in $S^{ \Lambda_n}$.
We will also consider elements of $S^\Lambda$ as configurations restricted to $\Lambda$.
We will simply write $\omega$ instead of $\omega_\Lambda$ since we will always make clear to which set $\omega$ belongs.
A pattern $p_n\in S^{ \Lambda_n}$ determines a cylinder set $[p_n]=\{\omega\in\Omega : \omega_{ \Lambda_n}=p_n\}$.
More generally, given $\Lambda\subset \Zd$ and $C\subseteq S^\Lambda$, let $[C]=\{\omega\in\Omega: \pi_\Lambda(\omega)\in C\}$
where $\pi_\Lambda$ is the projection from $\Omega$ onto $S^\Lambda$.

Finally, the shift action $\Theta=(\theta_x,x\in\Zd)$ is defined as usual: for each $x\in\Zd$, $\theta_x:\Omega\to \Omega$
and $(\theta_x\,\omega)_y=\omega_{y-x}$, for all $y\in\Zd$.
A probability measure $\nu$ on $(\Omega,\mathfrak{B})$ is shift invariant if $\nu\circ \theta_x=\nu$ for all $x\in\Z^d$.
We denote by $\EuScript{M}_\Theta(\Omega)$ the set of shift-invariant probability measures, which is a simplex.
A shift-invariant probability measure is ergodic if it is trivial on the $\sigma$-algebra of all shift-invariant
events $\{A\in \mathfrak{B}: \theta_x A=A\;\textup{for all}\;x\in\Z^d\}$. Ergodic measures are the extreme points of
$\EuScript{M}_\Theta(\Omega)$; we denote them by $\mathrm{ex}\,\EuScript{M}_\Theta(\Omega)$.


We now define what we mean by a Gaussian concentration bound.
Let $F:\Omega\to\R$ be a function and $x\in\Zd$. The oscillation of $F$ at $x$ is defined by
\[
\gdelta_x (F) = \sup\big\{ |F(\omega) - F(\omega')| : \omega,\omega'\in \Omega\;\textup{differ only at site}\;x\big\}.
\]
Given $\Lambda\Subset \Zd$ and
two configurations $\omega,\eta\in\Omega$ such that  $\omega_{\Lambda^c}=\eta_{\Lambda^c}$, one has
\[
|F(\omega)-F(\eta)| \leq \sum_{x\in\Lambda} \delta_x(F)\, \un_{\{\omega_x\neq\eta_x\}}.
\]
We introduce the space of local functions
\[
\mathscr{L}=\bigcup_{\Lambda \Subset \Z^d} \mathscr{L}_\Lambda
\]
where $F:\Omega\to\R$ belongs to $\mathscr{L}_\Lambda$
if there exists $\Lambda\Subset\Zd$ (the dependence set of $F$) such that for all $\omega,\widetilde\omega,\widehat\omega$,
$F(\omega_{\Lambda}\widetilde\omega_{\Lambda^c}) = F(\omega_{\Lambda}\widehat\omega_{\Lambda^c})$.
(This means that $F$ is $\mathfrak{B}_{\Lambda}$-measurable.)
Equivalently, this means that $\delta_x(F)=0$ for all $x\notin \Lambda$.
It is understood that $\Lambda$ is the smallest such set.
Local functions are continuous, hence bounded since $\Omega$ is compact. (In fact, continuous functions are obtained as uniform limits of local functions.)

We write $\ushort{\gdelta}(F)$ for the infinite array $(\gdelta_x(F), x\in\Zd)$, and let
\[
\|\ushort{\gdelta} (F)\|_2^2:=
\|\ushort{\gdelta} (F)\|_{\ell^2(\Zd)}^2= \sum_{x\in\Zd} \gdelta_x(F)^2.
\]
We use the notation $\E_\mu$ for the integration with respect to $\mu$.

\begin{definition}[Gaussian concentration bound]
\leavevmode\\
Let $\mu$ be a probability measure on $(\Omega,\mathfrak{B})$.
We say that it satisfies a Gaussian concentration bound if there exists $D=D(\mu)>0$ such that,
for all functions $F\in\mathscr{L}$, we have
\begin{equation}\label{eq:gemb}
\E_\mu \big[\exp\left(F - \E_\mu (F)\right)\big] \leq \exp\left(D\|\ushort{\delta} (F) \|_2^2\right).
\end{equation}
For the sake of brevity, we will say that $\mu$ satisfies $\gcb{D}$, or simply \textit{GCB}.
\end{definition}
A key point in this definition is that $D$ is independent of $F$, in particular it is independent of the size of the
dependence set of $F$.
\begin{remark}
Let us notice that if \eqref{eq:gemb} holds for all local functions, then it holds for a larger class of functions, namely
continuous functions such that $\|\ushort{\delta} (F)\|_2<+\infty$ (see \cite{CCR}).
\end{remark}
Inequality \eqref{eq:gemb} easily implies the following tail inequality that we will use several times.
\begin{proposition}\label{prop-exp-dev}
\leavevmode\\
If a probability measure $\mu$ on $(\Omega,\mathfrak{B})$ satisfies $\gcb{D}$ then, for all $u>0$,
\begin{align}
\label{eq:expcon1}
& \mu\big\{\omega\in \Omega : F(\omega)- \E_\mu (F)\geq u \big\} \leq
\exp\left(-\frac{u^2}{4D\|\ushort{\delta}(F) \|_2^2}\right).
\end{align}
\end{proposition}

\begin{proof}
If $F\in\mathscr{L}$ then, for any $\lambda>0$, $\lambda F$ obviously belongs to $\mathscr{L}$.
Applying Markov's inequality and \eqref{eq:gemb} we get
\begin{align*}
\mu\left\{\omega\in \Omega : F(\omega) - \E_\mu (F) \geq u \right\}
&\leq  \exp\left(-\lambda u\right)\ \E_\mu\big[\exp\left(\lambda(F - \E_\mu (F)\right) \big]\\
& \leq \exp\left(-\lambda u+D\|\ushort{\delta} (F) \|_2^2\, \lambda^2\right).
\end{align*}
Minimizing over $\lambda>0$ yields \eqref{eq:expcon1}.
\end{proof}
Observe that \eqref{eq:expcon1} can be applied to $-F$, which gives the same bound, and thus by a union bound
we get \eqref{eq:expcon1} with $|F(\omega) - \E_\mu (F)|$ by multiplying the bound by $2$.

The next result shows that a shift-invariant probability measure which satisfies GCB must be mixing.
\begin{proposition}\label{GCBimpliesergodicity}
\leavevmode\\
Let $\mu$ be a shift-invariant probability measure on $(\Omega,\mathfrak{B})$ which satisfies GCB.
Then $\mu$ is mixing.
\end{proposition}
\begin{proof}
First we remark that by Lemma \ref{lem:deltasums}, combined with Proposition \ref{prop-exp-dev}, we conclude that for every sequence $V_n, n\in \N$ of finite
subsets of $\Zd$ such that $|V_n|\to\infty$ as $n\to\infty$, and for every local function $f$, we have that 
$\frac{1}{|V_n|}\sum_{x\in V_n} \theta_x f$ converges to $\E_\mu(f)$ in $\mu$-probability as $n\to\infty$.

We then argue by contradiction. Assume that $\mu$ is not mixing. Then
there exist local functions $f, g$ (without loss of generality both
of $\mu$ expectation zero) and a sequence $x_n$, with $|x_n|\to\infty$ such that
\[
\int f \theta_{x_n} g \dd\mu
\]
does not converge to zero as $n\to\infty$. By locality, both functions $f,g$ are uniformly bounded, and
therefore the sequence $\int f \theta_{x_n} g \dd\mu, n\in \N$ is a bounded sequence. Therefore, there exists a subsequence
$y_n$ such that along that subsequence
\begin{equation}\label{bimbo}
\lim_{n\to\infty}\int f \theta_{y_n} g \dd\mu=a\not= 0.
\end{equation}
As a consequence,
\[
\lim_{n\to\infty}\int f \frac1n\sum_{k=0}^n\theta_{y_k} g \dd\mu=a\not=0.
\]
However, as we saw before,  $\frac1n\sum_{k=0}^n\theta_{y_k} g$ converges to zero in probability by GCB. Then
via dominated convergence we obtain
\[
\lim_{n\to\infty}\left|\int f \frac1n\sum_{k=0}^n\theta_{y_k} g \dd\mu\right|= 0
\]
which contradicts \eqref{bimbo}.
\end{proof}

\begin{remark}
A stronger property than mixing is tail-triviality of $\mu$, which in the context of Gibbs measures is equivalent with the
fact that $\mu$ is an extreme point of the set of Gibbs measures.
Though we strongly believe that GCB implies tail triviality and even more that $\mu$ is the unique Gibbs measure, at present we are not able to prove this.
\end{remark}

Let us finish this section by a variance inequality implied by \eqref{eq:gemb}. If $\mu$ satisfies $\gcb{D}$ then
\begin{equation}\label{devroye-ineq}
\mathrm{Var}_{\gmu}(F):=\E_{\mu}\big(F^2\big)-\E_{\mu}\big(F\big)^2\leq 2D \|\ushort{\delta} (F) \|_2^2
\end{equation}
for all functions $F\in\mathscr{L}$. The proof goes as follows. Take $\lambda>0$ and apply inequality
\eqref{eq:gemb} to $\lambda F$, subtract $1$ on both sides, and then divide out by $\lambda^2$ the resulting inequality. Then \eqref{devroye-ineq}
follows easily by Taylor expansion and letting $\lambda$ tend to $0$. (See \cite{CCR} for details.)

\section{Gibbs measures and equilibrium states}\label{sec:Gibbs}

\subsection{Potentials, specifications and relative entropy}

We refer to \cite{Geo} or \cite{FVbook} for details. We consider shift-invariant uniformly summable potentials.
More precisely, a potential is a family of functions $(\Phi(\Lambda,\cdot))_{\Lambda \Subset \Zd}$ such that,
for each (nonempty) $\Lambda\Subset\Zd$, the function $\Phi_\Lambda:\Omega\to\R$ is $\mathfrak{B}_\Lambda$-measurable.
Shift-invariance means that $\Phi(\Lambda+x,\theta_x\omega)=\Phi(\Lambda,\omega)$ for all $\Lambda\Subset\Zd$, $
\omega\in\Omega$ and $x\in\Zd$ (where $\Lambda+x=\{y+x: y\in\Lambda\}$).
Uniform summability is the property that
\[
\sum_{\substack{\Lambda\Subset\Zd\\ \Lambda \ni 0}}
\|\Phi(\Lambda,\cdot)\|_\infty <\infty.
\]
We shall denote by $\mathscr{B}_\Theta$ the space of uniformly summable shift-invariant potentials.

Given $\Phi\in \mathscr{B}_\Theta$ and $\Lambda\Subset\Zd$, the associated Hamiltonian in $\Lambda$ with
boundary condition $\eta\in\Omega$ is given by
\[
\mathcal{H}_{\Lambda}(\omega|\eta)= \sum_{\substack{\Lambda'\Subset\Zd\\ \Lambda'\cap\, \Lambda \,\neq\,\emptyset}} \Phi(\Lambda',\omega_{\Lambda}\eta_{\Lambda^{\!c}})\,.
\]
The corresponding Gibbsian specification is then defined as
\[
\bgamma^{\Phi}_{\Lambda} (\omega|\eta)
=\frac{\exp\left(-\mathcal{H}_{\Lambda}(\omega|\eta)\right)}{Z_{\Lambda}(\eta)}
\]
where $Z_{\Lambda}(\eta)$ is the partition function in $\Lambda$ (normalizing factor).
We say that $\mu$ is a Gibbs measure for the potential $\Phi$ if, for every  $\Lambda \Subset \Zd$,
$\bgamma^{\Phi}_{\Lambda} (\omega|\cdot)$ is a version of the conditional probability
$\mu(\omega_{\Lambda}| \mathfrak{B}\!_{\Lambda^c})$.
Equivalently, this means that for all $A\in \mathfrak{B}$, $\Lambda \Subset \Zd$, one has the so-called ``DLR equations''
\[
\mu(A)=\int  \sum_{\omega'\in S^\Lambda} \bgamma^{\Phi}_{\Lambda} (\omega'|\eta) \, \un\!_{A}(\omega'_{\Lambda}\eta_{\Lambda^c})
\dd\mu(\eta).
\]
The set of Gibbs measures for a given potential $\Phi$, denoted by $\EuScript{G}(\Phi)$, is never empty, but it may be not reduced to a
singleton. It is a simplex.
The set $\EuScript{G}_\Theta(\Phi):=\EuScript{G}(\Phi)\cap \EuScript{M}_\Theta(\Omega)$, that is, the set of shift-invariant Gibbs measures for
$\Phi$, is never empty.
It is a simplex whose set of extreme points, denoted by
$\mathrm{ex}\, \EuScript{G}_\Theta(\Phi)$, coincides with the set of ergodic Gibbs measures for $\Phi$, that is,
$\mathrm{ex}\, \EuScript{G}_\Theta(\Phi)=\EuScript{G}_\Theta(\Phi)\cap\mathrm{ex}\, \EuScript{M}_\Theta(\Omega)$.
Of course, when $\EuScript{G}(\Phi)$ is a singleton, then the unique Gibbs measure is shift-invariant and ergodic.

We now define relative entropy which plays a central role in this paper.
Let $\mu,\nu\in\EuScript{M}_\Theta(\Omega)$. For each $n\in\N$, we denote by $\nu_n$ (resp. $\mu_n$) the probability
measure induced by $\nu$ (resp. $\mu$) on $S^{ \Lambda_n}$ by projection.
Then the relative entropy of $\nu_n$ with respect to $\mu_n$ is defined by
\[
H_n(\nu|\mu)=\sum_{p_n\in S^{ \Lambda_n}} \nu([\,p_n]) \log \frac{\nu([\,p_n])}{\mu([\,p_n])}.
\]
We denote by $\log$ the natural logarithm.
\begin{definition}[Relative entropy]
\leavevmode\\
Let $\mu,\nu\in\EuScript{M}_\Theta(\Omega)$. The lower and upper relative entropies of $\nu$ with respect to $\mu$ are
\begin{equation}\label{def-hdown-hup}
h_*(\nu|\mu)=\liminf_{n\to+\infty} \frac{H_n(\nu|\mu)}{(2n+1)^d}
\quad\textup{and}\quad
h^*(\nu|\mu)=\limsup_{n\to+\infty} \frac{H_n(\nu|\mu)}{(2n+1)^d}.
\end{equation}
When the limit exists, we put $h(\nu|\mu)=h_*(\nu|\mu)=h^*(\nu|\mu)$.
\end{definition}

It is well-known that $h_*(\nu|\mu)$ and $h^*(\nu|\mu)$ are nonnegative numbers.
Let $\Phi\in \mathscr{B}_\Theta$ and $\mu\in \EuScript{G}_\Theta(\Phi)$. It is proved in \cite[Chapter 15]{Geo}
that for any $\nu\in \EuScript{M}_\Theta(\Omega)$ we have
\begin{equation}\label{formula-RE}
h_*(\nu|\mu)=h^*(\nu|\mu)
=h(\nu|\mu)=P(\Phi)+\sum_{\Lambda\ni \,0}|\Lambda|^{-1}\E_{\nu}[\Phi(\Lambda,\cdot)]-h(\nu)
\end{equation}
where $P(\Phi)$ is the pressure of $\Phi$ and $h(\nu)$ is the entropy of $\nu$. (We do not need to give the precise definitions of these two
quantities since they are not used explicitly in the sequel.)
Notice that, given $\nu\in \EuScript{M}_\Theta(\Omega)$, $h(\nu|\mu)$ is the same number for all $\mu\in\EuScript{G}_\Theta(\Phi)$, so it is natural to define, for each
$\nu\in \EuScript{M}_\Theta(\Omega)$,
\[
h(\nu|\Phi):=h(\nu|\mu)\quad \text{where } \mu\text{  is any element in} \;\EuScript{G}_\Theta(\Phi).
\]
We recall the definition of equilibrium states.
\begin{definition}[Equilibrium states]
\leavevmode\\
Let $\Phi\in \mathscr{B}_\Theta$.
A shift-invariant probability measure $\nu$ such that $h(\nu|\Phi)=0$ is called an equilibrium state for $\Phi$.
\end{definition}
We have the following fundamental result which is usually referred to as the variational principle for equilibrium states.
\begin{theorem}[{\cite[Chapter 15]{Geo}}, {\cite[Theorem 4.2]{ruelle2004}}]
\label{thm:vp}
\leavevmode\\
Let $\Phi\in \mathscr{B}_\Theta$.
We have $h(\nu|\Phi)=0$ if and only if $\nu\in \EuScript{G}_\Theta(\Phi)$.
In particular,  $\EuScript{G}_\Theta(\Phi)$ coincides with the set of equilibrium states for $\Phi$.
\end{theorem}

\subsection{Examples}

\subsubsection{Dobrushin's uniqueness}

We first recall that a sufficient condition for GCB to hold is Dobrushin's uniqueness condition which guarantees that $\Phi\in\mathscr{B}_\Theta$
admits a unique Gibbs measure. Under this condition, this Gibbs measure satisfies a Gaussian concentration bound with a constant equal to
$(2(1-\mathfrak{c}(\bgamma^{\Phi}))^2)^{-1}$, where $\mathfrak{c}(\bgamma^{\Phi})<1$. This holds for instance for any finite-range potential $
\beta\Phi$ provided that $\beta>0$ is small enough. A basic example is the 
nearest-neighbor ferromagnetic Ising model. Note that if one adds a external uniform magnetic field $h$ in this model, then Dobrushin's uniqueness condition can hold for every $\beta>0$ if $|h|$ is large enough, hence we can have GCB at low temperature.
We refer to \cite[Chapter 8]{Geo} for more details and examples.

\subsubsection{Ising model}

The nearest-neighbor ferromagnetic Ising model on $\Z^d$ with zero external magnetic field is a basic example illustrating the above definitions
and results. Take $S=\{-1,+1\}$, $d\geq 2$, and
\[
\Phi(\Lambda,\omega)=
\begin{cases}
- \,\omega_x\,\omega_y & \textup{if}\quad \Lambda=\{x,y\}\;\textup{and}\;\|x-y\|_1=1\\
0 & \textup{otherwise}.
\end{cases}
\]
Denote by $\mu_{\beta\Phi}^+$ (resp. $\mu_{\beta\Phi}^-$) is the Gibbs measure obtained with the ``$+$'' boundary condition (resp. the ``$-$''
boundary condition).
It is well-known that there exists $\beta_c=\beta_c(d)$ such that, for all $\beta>\beta_c$, $\mu_{\beta\Phi}^+$ and $\mu_{\beta\Phi}^-$ are
distinct ergodic equilibrium states. According to our main theorem, they cannot satisfy GCB.

For $d=2$, there is an inverse temperature $\beta_c$ such that, for all $\beta\leq \beta_c$, $\EuScript{G}_\Theta(\beta\Phi)$ is a singleton, and, for
all $\beta>\beta_c$, $\EuScript{G}_\Theta(\beta\Phi)=\{\lambda\mu_{\beta\Phi}^+ + (1-\lambda)\mu_{\beta\Phi}^-:\lambda\in\left[0,1\right]\}$
(hence $\mathrm{ex}\, \EuScript{G}_\Theta(\beta\Phi)=\{\mu_{\beta\Phi}^+,\mu_{\beta\Phi}^-\}$).
It is known (see \cite[p. 172]
{ellis}) that for $\beta=\beta_c=\log(1+\sqrt{2})/2$, there is a unique Gibbs measure $\mu_{\beta_c}$ such that
\begin{equation}\label{variance-critical-Ising}
\lim_{n\to\infty}\frac{1}{|\Lambda_n|} \mathrm{Var}_{\gmu_{\beta_c}}\!\!\left(\,\sum_{x\in\Lambda_n} \omega_x \right)=+\infty.
\end{equation}
Therefore, $\mu_{\beta_c}$ cannot satisfy a Gaussian concentration bound because, by \eqref{devroye-ineq}, one
would have
\begin{equation}\label{pouac}
\limsup_{n\to\infty}\frac{1}{|\Lambda_n|}\mathrm{Var}_{\gmu_{\beta_c}}\!\!\left(\,\sum_{x\in\Lambda_n} \omega_x \right) \leq 8D
\end{equation}
which contradicts \eqref{variance-critical-Ising}. To obtain \eqref{pouac}, apply \eqref{devroye-ineq} to
$F(\omega)=\sum_{x\in\Lambda_n} \pi_{\{0\}}(\theta_x\, \omega)$, where $\pi_{\{0\}}(\omega)=\omega_0$.
Then use Lemma \ref{lem:deltasums} to get
\[
\mathrm{Var}_{\gmu_{\beta_c}}\!\!\left(\,\sum_{x\in\Lambda_n} \omega_x \right) \leq 8D |\Lambda_n|
\]
for all $n$. Therefore, the Gibbs measure for the 2D Ising model at critical temperature does not satisfy a Gaussian concentration bound. It does
not even satisfy the variance inequality \eqref{devroye-ineq}.
\begin{remark}
Let us mention that GCB holds for all $\beta<\beta_c$ in the 2D Ising model. This follows by combining several results
taken from \cite{MOS,SZ1,SZ2}. They involve log-Sobolev inequality, `complete analyticity' in the sense of Dobrushin and Shlosman, and a 
particular phenomenon arising only in  dimension  two for finite-range potentials. We refer to \cite{moles} for details.  This  example suggests 
that GCB is equivalent to complete analyticity in any dimension  and for any finite-range potential, but we are not able to prove this at present.
\end{remark}

For $d=3$, the situation is more complicated at low temperatures.
Indeed, in addition to $\mu_{\beta\Phi}^+$ and $\mu_{\beta\Phi}^-$, $\EuScript{G}(\beta\Phi)$ also contains, for $\beta$ large enough, a family
of Gibbs measures which are not shift-invariant, the so-called Dobrushin states. In other words,
$\EuScript{G}(\beta\Phi)\backslash \EuScript{G}_\Theta(\beta\Phi)\neq \emptyset$.
In the present paper, we do not deal with these non-shift invariant Gibbs measures for the Ising model (and other models).
Indeed, whereas for translation invariant Gibbs states we obtain here a general uniqueness result, the situation becomes much more intricate for non-translation Gibbs states, and even more for non-translation invariant potentials.
In fact, Dobrushin interface states can be shown to be incompatible with GCB, using the volume large deviation bound \eqref{voldev}.
However, other more subtle scenarios of non-uniqueness combined with a unique translation invariant Gibbs measure can occur, such
as in \cite{Borgs}, and with the techniques developed in this paper, we cannot show that GCB excludes such scenarios of non-uniqueness.

\subsubsection{Dyson model}

Consider the Dyson model: $S=\{-1,+1\}$, $d=1$, and
\[
\Phi(\Lambda,\omega)=
\begin{cases}
- \frac{\omega_x\,\omega_y}{|x-y|^{\alpha}} & \textup{if}\quad \Lambda=\{x,y\}\;\textup{such that}\;x\neq y\\
0 & \textup{otherwise}
\end{cases}
\]
with $\alpha>1$. As in the Ising model, consider the Gibbs measures $\mu_{\beta\Phi}^+$ and  $\mu_{\beta\Phi}^-$ obtained as the infinite-volume
limits of the corresponding specification with the ``$+$'' and the ``$-$'' boundary conditions, respectively. Let $1<\alpha\leq 2$. There exists $
\beta_c>0$ such that, for all $\beta<\beta_c$,
$\mu_{\beta\Phi}^+\neq \mu_{\beta\Phi}^-$, and $\mathrm{ex}\, \EuScript{G}_\Theta(\beta\Phi)=\{\mu_{\beta\Phi}^+,\mu_{\beta\Phi}^-\}$.
We refer to \cite{le-ny-survey} for the relevant references.
By the above theorem, these two equilibrium states cannot satisfy a Gaussian concentration bound.

\subsubsection{Other examples}

For instance, consider the Potts model with $S=\{1,\ldots,N\}$ for sufficiently large $N$, for which
there exists $0<\beta_{{\scriptscriptstyle N}}<+\infty$ such that $\big|\mathrm{ex}\, \EuScript{G}_\Theta(\beta_{{\scriptscriptstyle N}} \Phi)|= N+1$,
one of these measures being symmetric under spin flip, and $\big|\mathrm{ex}\, \EuScript{G}_\Theta(\beta \Phi)|= N$ when $
\beta>\beta_{{\scriptscriptstyle N}}$.
Hence, there are $N+1$ ergodic equilibrium states in this model which, at inverse temperature
$\beta_{{\scriptscriptstyle N}}$, cannot satisfy a Gaussian concentration bound, and when $\beta>\beta_{{\scriptscriptstyle N}}$,
there are $N$ such ergodic equilibrium states. See \cite[Chapter 19]{Geo} for details and other examples.

\section{Proof of the main result}

\subsection{An abstract result}

The first way to prove the main theorem is to establish an abstract theorem which is of independent interest, and can be applied to
equilibrium states.
To state it, we need to define the ``positive relative entropy property''.

\begin{definition}[Positive relative entropy property]
\label{def:PREP}
\leavevmode\\
An ergodic probability measure $\mu$ on $\Omega$ is said to have the positive relative entropy property
if $h_*(\nu|\mu)>0$ for every ergodic probability measure $\nu\neq \mu$.
\end{definition}
The lower relative entropy $h_*(\nu|\mu)$ is defined in \eqref{def-hdown-hup}.
We can now state the following result.
(Recall that by Proposition \ref{GCBimpliesergodicity} we only need to consider \textit{ergodic} measures.)

\begin{theorem}[GCB implies positive relative entropy]
\label{abstracttheorem}
\leavevmode\\
Let $\mu$ be an ergodic probability measure on $\Omega$ which satisfies GCB. Then $\mu$ has the
positive relative entropy property.
\end{theorem}

\begin{sketchofproof}
We outline the proof which consists in the following three steps.
We first prove that GCB implies the so-called ``blowing-up'' property (Section \ref{GCBimpliesBup}).
Then we prove that the blowing-up property implies the ``exponential rate of convergence for frequencies'' (Section \ref{BupimpliesERCF}).
Finally we prove that the latter implies the positive relative entropy property (Section \ref{ERCFimpliesPDP}).
\end{sketchofproof}

Now the main theorem stated in the introduction is a corollary of the previous theorem.
\begin{corollary}
\leavevmode\\
Let $\Phi\in \mathscr{B}_\Theta$. If $|\mathrm{ex}\, \EuScript{G}_\Theta(\Phi)|\geq 2$
and $\mu \in \mathrm{ex}\, \EuScript{G}_\Theta(\Phi)$, then $\mu$ cannot satisfy a Gaussian concentration bound.
\end{corollary}
\begin{proof}
Fix an arbitrary $\mu\in\mathrm{ex}\, \EuScript{G}_\Theta(\Phi)$.
By assumption there exists $\mu'\in \mathrm{ex}\, \EuScript{G}_\Theta(\Phi)$ such that $\mu\neq \mu'$.
But, by Theorem \ref{thm:vp}, we have $h(\mu'|\Phi)=h_*(\mu'|\mu)=0$, whence $\mu$ does not have
the positive relative entropy property. Therefore, according to Theorem \ref{abstracttheorem}, $\mu$
cannot satisfy GCB. This proves that none of the elements of $\mathrm{ex}\, \EuScript{G}_\Theta(\Phi)$ can satisfy GCB.
\end{proof}

\subsection{GCB implies blowing-up}\label{GCBimpliesBup}

Let $\Lambda\Subset\Zd$ (finite subset of $\Zd$). We define the (non-normalized)
Hamming distance between two configurations $\omega$ and $\eta$ in $S^\Lambda$ by
\[
\db_\Lambda(\omega,\eta)=\sum_{x\in \Lambda} \un_{\{\omega_x\neq\eta_x\}}\in\{0,1,\ldots,|\Lambda|\}.
\]
So we count at how many sites the configurations $\omega$ and $\eta$ in $S^\Lambda$ differ.
Clearly, we can see $\db_\Lambda$ as a local function on $\Omega\times\Omega$.
Given a subset $C\subseteq S^\Lambda$ define
\[
\db_\Lambda(\omega,C)=\inf_{\omega'\in C}\db_\Lambda(\omega,\omega').
\]
Given $\varepsilon\in [0,1]$, define the $\varepsilon$-neighborhood (or $\varepsilon$-blow-up) of $C$ as
\[
\langle C\rangle_\varepsilon=\big\{\omega\in S^\Lambda :
\db_\Lambda(\omega,C)< \varepsilon|\Lambda|\big\}\subseteq S^\Lambda.
\]
Recall that if $C\subseteq S^\Lambda$, $[C]=\{\omega\in\Omega: \pi_\Lambda(\omega)\in C\}$
where $\pi_\Lambda$ is the projection from $\Omega$ onto $S^\Lambda$.
We now define the blowing-up property.
\begin{definition}[Blowing-up property]\label{def:Bup}
\leavevmode\\
An ergodic probability measure $\mu$ on $(\Omega,\mathfrak{B})$ has the blowing-up property if given $
\varepsilon>0$ there is a $\delta>0$ and an $N$ such that if $n\geq N$ and $C\subseteq S^{ \Lambda_n}$ then
\[
\mu([C])\geq \e^{-(2n+1)^d\delta}\quad\textup{implies}\quad
\mu(\langle C\rangle_\varepsilon)\geq 1-\varepsilon.
\]
\end{definition}
Obviously, the blowing-up property can be formulated in terms of finite subsets of $\Z^d$ of arbitrary shape,
instead of cubes, but we will not need this generalization. The blowing-up property roughly says that any collection of configurations on a large finite box which has a
total measure which is not too exponentially small is such that most configurations are close to this collection in the Hamming distance.
We have the following result.
\begin{proposition}
\leavevmode\\
Let $\Lambda\Subset \Zd$ and $C\subseteq S^\Lambda$.
Suppose that $\mu$ is a probability measure which satisfies $\gcb{D}$
and such that $\mu([C])>0$. Then, we have
\begin{equation}\label{eq:concen-H}
\mu\big(\langle C\rangle_\varepsilon\big)\geq
1-\exp\left[-\frac{|\Lambda|}{4D}\left(\varepsilon-\frac{2\sqrt{D \log(\mu([C])^{-1})}}{\sqrt{|\Lambda|}}\,\right)^2\, \right]
\end{equation}
whenever $\varepsilon>2\sqrt{D\log(\mu([C])^{-1})/|\Lambda|}$. In particular, $\mu$ satisfies the blowing-up
property.
\end{proposition}
Note that we do not require $\mu$ to be shift-invariant. As already mentioned, if $\mu$ is taken shift-invariant then it must be ergodic, this is enforced by the Gaussian concentration bound.

\begin{proof}
Consider the local function $F(\omega)=\db_\Lambda(\omega,C)$. One easily checks that
$\delta_x(F)\leq 1$ for all $x\in \Lambda$. Applying \eqref{eq:expcon1}  gives
\begin{equation}\label{puic}
\mu\left\{\omega\in \Omega : F(\omega) \geq u+\E_\mu (F)\right\} \leq
\exp\left(-\frac{u^2}{4D|\Lambda|}\right)
\end{equation}
for all $u>0$. We now bound $\E_\mu (F)$ from above. Applying \eqref{eq:gemb} to $-\lambda F$, for some $\lambda>0$ to be
chosen later on, we get
\[
\exp\left(\lambda\E_\mu(F)\right) \E_\mu\big[\exp\left(-\lambda F\right) \big]
\leq \exp\left(D\lambda^2 |\Lambda|\right).
\]
Observe that by definition of $F$ we have
\[
\E_\mu\big[\exp\left(-\lambda F\right) \big]\geq
\E_\mu\big[\un_{C}\exp\left(-\lambda F\right) \big]=\mu([C]).
\]
Combining the two previous inequalities and taking the logarithm gives
\[
\E_\mu(F)\leq \inf_{\lambda>0} \Big\{D \lambda|\Lambda| +
\frac{1}{\lambda}\ln \big(\mu([C])^{-1}\big)\Big\}
\]
which gives the following estimate by taking the value of $\lambda$ minimizing the right-hand side
\[
\E_\mu(F)\leq 2\sqrt{D|\Lambda| \ln\big(\mu([C])^{-1}\big)}=:v_0.
\]
Therefore inequality \eqref{puic} implies that
\[
\mu\left\{\omega\in \Omega : F(\omega) \geq v\right\} \leq
\exp\left(-\frac{(v-v_0)^2}{4D|\Lambda|}\right)
\]
for all $v>v_0$. To finish the proof of \eqref{eq:concen-H}, take $v=\varepsilon|\Lambda|$ and observe that, by definition
of $F$, $\mu\left\{\omega\in \Omega : F(\omega) \geq v\right\}=\mu\big(\Omega\backslash \langle C\rangle_\varepsilon\big)$.\newline
We now prove that $\mu$ satisfies the blowing-up property.
Fix $\varepsilon>0$ arbitrarily, and take $C$ such that $\mu([C])\geq \e^{-(2n+1)^d \delta}$ for
some $\delta>0$ to be chosen later on, subject to the condition $\varepsilon>2\sqrt{D\delta}$.
We now apply \eqref{eq:concen-H} with $\Lambda= \Lambda_n$ to get
\[
\mu(\langle C\rangle_\varepsilon)\geq
1-\exp\left(-\frac{(2n+1)^d}{4D} \Big(\varepsilon-2\sqrt{D\delta}\,\Big)^2\right).
\]
Taking $\delta=\varepsilon^2/(9D)$ yields
\[
\mu(\langle C\rangle_\varepsilon)\geq 1-\varepsilon
\]
for all $n\geq N:=\lfloor (36D\varepsilon^{-2}\log \varepsilon^{-1})^{1/d}\rfloor/2$.
\end{proof}

\subsection{Blowing-up implies exponential rate for frequencies}\label{BupimpliesERCF}

Given $\omega\in \Omega$, $n>k\geq 0$, and a pattern $p_k\in S^{ \Lambda_k}$, let
\begin{equation}\label{def:empfreq}
\mathfrak{f}_{n,k}(\omega;p_k)=\frac{|\{x\in  \Lambda_{n-k} : (\theta_x\, \omega)_{ \Lambda_k}=p_k\}|}{(2(n-k)+1)^d}\,.
\end{equation}
In words, this is simply the frequency of occurrence of the pattern $p_k$ if we look at the configuration $\omega$ restricted to
the cube $ \Lambda_n$. Let $\mu$ be an ergodic probability measure.
By the multidimensional ergodic theorem, for $\mu$-almost every $\omega$,
we have
\[
\lim_{n\to\infty}\mathfrak{f}_{n,k}(\omega;p_k)=\mu([p_k]).
\]
Given two probability measures $\mu$ and $\nu$ on $\Omega$, recall that $\nu_k$ (resp. $\mu_k$) is the probability
measure induced by $\nu$ (resp. $\mu$) on $S^{ \Lambda_k}$ by projection.
The total variation distance between $\mu_k$ and $\nu_k$ is defined by
\[
\|\mu_k-\nu_k\|_{\tv}=\frac{1}{2}\sum_{p_k\in S^{ \Lambda_k}} |\mu([p_k])-\nu([p_k])|.
\]
We can now define the property of exponential rate of convergence for frequencies.
\begin{definition}[Exponential rate of convergence for frequencies]
\leavevmode\\
An ergodic probability measure $\mu$ on $(\Omega,\mathfrak{B})$ has the exponential rate of convergence
property for frequencies if, given $k$ and $\varepsilon>0$, there is a $\delta>0$ and an $N$ such that
\[
\mu\left\{\omega\in\Omega:  \|\mathfrak{f}_{n,k}(\omega;\cdot)-\mu_k\|_{\tv}\geq \varepsilon\right\}
\leq \e^{-(2n+1)^d \delta},\;\forall n\geq N.
\]
\end{definition}

We now prove that this property is implied by the blowing-up property.
\begin{proposition}
Let $\mu$ be an ergodic probability measure on $(\Omega,\mathfrak{B})$.
If $\mu$ has the blowing-up property, then it has the exponential rate of convergence
property for frequencies.
\end{proposition}
\begin{proof}
We adapt a proof of \cite{martonshields} to our setting.
Take $\varepsilon>0$ and $k\geq 0$, and for any $n\geq 0$ let
\[
\EuScript{B}(n,k,\varepsilon)=\left\{\omega\in\Omega:  \|\mathfrak{f}_{n,k}(\omega;\cdot)-\mu_k\|_{\tv}\geq \varepsilon\right\}.
\]
Note that we can naturally identify this subset of $\Omega$ with a subset of $S^{ \Lambda_n}$.
We have the following lemma whose proof is given in Appendix \ref{appendix:shields} below.
\begin{lemma}\label{lemma-shields}
Let $\varepsilon>0$ and $k\geq 0$, and define
\[
\varrho=\frac{2\varepsilon}{5(2k+1)^d}.
\]
There exists $\breve{N}>k$ such that, if $n\geq \breve{N}$ and $\db_{ \Lambda_n}(\omega,\eta)\leq \varrho\,(2n+1)^d$, then
\[
\|\mathfrak{f}_{n,k}(\omega;\cdot)-\mathfrak{f}_{n,k}(\eta;\cdot)\|_{\tv}\leq \frac{\varepsilon}{2}.
\]
\end{lemma}
The lemma implies that
\begin{equation}\label{pluie}
\langle \EuScript{B}(n,k,\varepsilon) \rangle_{\varrho}\subseteq \EuScript{B}\Big(n,k,\frac{\varepsilon}{2}\Big).
\end{equation}
Now, by the blowing-up property (Definition \ref{def:Bup}), there is a $\delta>0$ and an $N$ such that, if
$n\geq N$ and $C\subseteq S^{ \Lambda_n}$ so that $\mu([C])\geq \e^{-(2n+1)^d\delta}$, then
$\mu(\langle C\rangle_\varrho)\geq 1-\varrho\geq 1-\varepsilon$. Therefore, if $n\geq N$ and if we suppose
that $\mu\big(\EuScript{B}(n,k,\varepsilon) \big)\geq \e^{-(2n+1)^d \delta}$, then we get
\[
\mu\big(\langle\EuScript{B}(n,k,\varepsilon)\rangle_\varrho \big)\geq 1-\varepsilon.
\]
This bound together with the inclusion \eqref{pluie} implies that for all $n\geq \max(N,\breve{N})$
\[
\mu\Big( \EuScript{B}\Big(n,k,\frac{\varepsilon}{2}\Big)\Big)\geq 1-\varepsilon.
\]
But this contradicts the multidimensional ergodic theorem which ensures that, for each $\varepsilon>0$ and each $k$,
\[
\lim_{n\to\infty}  \mu\Big( \EuScript{B}\Big(n,k,\frac{\varepsilon}{2}\Big)\Big)=0.
\]
Indeed, we have
\[
\mu\Big( \EuScript{B}\Big(n,k,\frac{\varepsilon}{2}\Big)\Big)
\leq \sum_{p_k \in S^{\Lambda_k}} \mu\left\{\omega\in\Omega : |\mathfrak{f}_{n,k}(\omega;p_k)-\mu([p_k])|\geq
\frac{\varepsilon}{|S|^{(2k+1)^d}} \right\}
\]
and each term in this finite sum goes to $0$ as $n\to+\infty$.
The proposition is proved.
\end{proof}

\subsection{Exponential rate for frequencies implies positive relative entropy}\label{ERCFimpliesPDP}

We now have the following proposition.
\begin{proposition}\label{prop:dimanche}
Let $\mu$ be an ergodic probability measure on $(\Omega,\mathfrak{B})$. If it has the
exponential rate of convergence property for frequencies, then, given $\varepsilon>0$ and $k$, there
is a $\delta>0$ such that, if $\nu$ is an ergodic probability measure such that
$h_*(\nu|\mu)<\delta$ then $\|\nu_k-\mu_k\|_{\tv}<\varepsilon$.
\end{proposition}
\begin{proof}
We adapt a proof of \cite{martonshields}.
Let $\rho$ be a positive number that we will specify later on, and suppose that
\[
\frac{H_n(\nu|\mu)}{(2n+1)^d}<\frac{\rho^2}{2}.
\]
If $n$ is large enough, we have by Markov's inequality\footnote{Let $f\geq 0$ be an integrable function on a probability
space $(\mathcal{X},\Sigma,\mu)$. If $\int f \dd\mu\leq \rho^2$ then $f(x)\leq \rho$, except for a set
of measure at most $\rho$.} and Lemma \ref{lem:boundrelent} (see appendix below) that there exists
a set $G_n\subseteq S^{ \Lambda_n}$ such that
\[
\nu([G_n])>1-\rho
\]
and
\begin{equation}\label{eq:patisson}
\e^{-(2n+1)^d \rho} \nu([\omega_{ \Lambda_n}])\leq \mu([\omega_{ \Lambda_n}])
\leq \e^{(2n+1)^d \rho} \nu([\omega_{ \Lambda_n}]),\quad \omega\in [G_n].
\end{equation}
Now, for each fixed $k$, and for $n$ large enough, the multidimensional ergodic theorem applied to $\nu$ tells us
that there is a set $\widetilde{G}_n\subseteq G_n$ such that
\begin{equation}\label{eq:courge}
\nu\big([\widetilde{G}_n]\big)>1-2\rho
\end{equation}
\begin{equation}\label{eq:potimaron}
\|\mathfrak{f}_{n,k}(\omega;\cdot)-\nu_k\|_{\tv}< \rho,\quad \omega\in [\widetilde{G}_n].
\end{equation}
If $\mu$ has the exponential rate of convergence property for frequencies then
\[
\mu\left\{\omega\in\Omega:  \|\mathfrak{f}_{n,k}(\omega;\cdot)-\mu_k\|_{\tv}\geq \rho\right\}\leq \e^{-(2n+1)^d \tau}
\]
for some $\tau>0$ and all $n$ sufficiently large.
Now, conditions \eqref{eq:patisson} and \eqref{eq:courge} imply that
\[
\mu\big([\widetilde{G}_n]\big)>(1-2\rho)\e^{-(2n+1)^d \rho}.
\]
Therefore, if $\rho$ is small enough and $n$ large enough, there exists $\omega\in [\widetilde{G}_n]$
such that $\|\mathfrak{f}_{n,k}(\omega;\cdot)-\mu_k \|_{\tv}<\rho$.
Indeed, it is enough to check that
\[
\mu\big([\widetilde{G}_n]\big)>\mu\left\{\omega\in\Omega:  \|\mathfrak{f}_{n,k}(\omega;\cdot)-\mu_k\|_{\tv}\geq \rho\right\}
\]
which implies that
\[
\big[\widetilde{G}_n\big] \cap \left\{\omega\in\Omega:  \|\mathfrak{f}_{n,k}(\omega;\cdot)-\mu_k\|_{\tv}\geq \rho\right\}
\]
has strictly positive $\mu$-measure.
Since \eqref{eq:potimaron} holds for this same
$\omega$, we thus arrive by the triangle inequality at the estimate $\|\nu_k-\mu_k\|_{\tv}<2\rho$.
Therefore, we proved that, given $\varepsilon$ and $k$, if  $\nu$ is ergodic such that $h_*(\nu|\mu)<\varepsilon^2/8$, then $\|\nu_k-\mu_k\|_{\tv}<\varepsilon$,
which ends the proof.
\end{proof}

Now we can state the main proposition of this section.
\begin{proposition}
If $\mu$ has the exponential rate of convergence property for frequencies then it has the positive relative entropy property.
\end{proposition}
\begin{proof}
Suppose that $\mu$ has not the positive relative entropy property, but satisfies the exponential rate of convergence property for
frequencies. We will obtain a contradiction. By assumption there is an ergodic probability measure $\widehat{\nu}\neq
\mu$ such that $h_*(\widehat{\nu}|\mu)=0$, and there is a $\hat{k}$ and an $\widehat{\varepsilon}>0$ such that
$\|\widehat{\nu}_{\hat{k}}-\mu_{\hat{k}}\|_{\tv}\geq \widehat{\varepsilon}$. Now we apply Proposition \ref{prop:dimanche}.
By the exponential rate of convergence property for frequencies, given $\hat{\varepsilon}$ and $\hat{k}$, there is a $\hat{\delta}>0$
such that if $\nu$ satisfies $h_*(\nu|\mu)<\hat{\delta}$ then $\|\nu_{\hat{k}}-\mu_{\hat{k}}\|_{\tv}<\widehat{\varepsilon}$.
We can take $\nu=\widehat{\nu}$, hence we arrive at $\|\widehat{\nu}_{\hat{k}}-\mu_{\hat{k}}\|_{\tv}< \widehat{\varepsilon}$, a
contradiction.
\end{proof}

\section{A proof via large deviations}

We present another proof of the main theorem, based on large deviations.
Suppose that $\mu$ is an ergodic equilibrium state which satisfies a Gaussian concentration bound.
Now assume that there exists $\mu'\in \mathrm{ex}\, \EuScript{G}_\Theta(\Phi)$ such that $\mu\neq \mu'$. We are going to
arrive at a contradiction. Notice that we can suppose that $\mu'$ is ergodic without loss of generality, because of the ergodic decomposition
\cite[Theorem 14.17, p. 298]{Geo}, and the fact that the map $\nu\mapsto h(\nu|\Phi)$, $\nu\in \EuScript{M}_\Theta(\Omega)$, is affine (which is
a consequence of \cite[Theorem 15.20, p. 318]{Geo} and \eqref{formula-RE}).

Since $\mu\neq \mu'$, there exists a local function $f:\Omega\to\R$ such that $\E_\mu(f) \neq \E_{\mu'}(f)$.
Without loss of generality, assume that $\E_{\mu'}(f)>\E_{\mu}(f)$. So there exists $\varepsilon>0$ such that
\[
\E_{\mu}(f)+\varepsilon=\E_{\mu'}(f).
\]
We want to apply Proposition \ref{prop-exp-dev} to $F=S_n f$, where, for each $n\geq 0$, $S_n f=\sum_{x\in\Lambda_n}
f\circ\theta_x$. We claim that
\[
\|\ushort{\gdelta} (S_n f)\|_2^2\leq (2n+1)^d \|\ushort{\gdelta}(f)\|_1^2
\]
where $\|\ushort{\gdelta}(f)\|_1:=\sum_{x\in\Zd} \delta_x(f)$ (which is finite since $f$ is local). See Appendix \ref{est}
for the proof.
Letting
\[
\mathcal{E}_{n,\varepsilon}:=\left\{ \omega\in\Omega : \frac{S_n f(\omega)}{(2n+1)^d}\geq \E_{\mu}(f)+\frac{\varepsilon}{3} \right\}.
\]
we get
\begin{equation}\label{voldev}
\mu(\mathcal{E}_{n,\varepsilon})\leq\exp\left(-\frac{(2n+1)^d \varepsilon^2}{36D \|\ushort{\gdelta}(f)\|_1^2}\right).
\end{equation}
Now let
\[
\mathcal{E}'_{n,\varepsilon}:=\left\{ \omega\in\Omega :
\frac{S_n f(\omega)}{(2n+1)^d}\in \left]\E_{\mu'}(f)-\frac{\varepsilon}{3},\E_{\mu'}(f)+\frac{\varepsilon}{3} \right[\right\}
\]
Since
\[
\mathcal{E}'_{n,\varepsilon}\subset \mathcal{E}_{n,\varepsilon}
\]
we deduce that
\begin{equation}\label{candya}
\limsup_{n\to+\infty} \frac{1}{(2n+1)^d}\log\mu(\mathcal{E}'_{n,\varepsilon}) \leq -\frac{\varepsilon^2}{36D \|\ushort{\gdelta}(f)\|_1^2}<0.
\end{equation}
Now we use the large deviation principle satisfied by $\mu$ (see \cite[Section 15.5]{Geo}) which implies that
\begin{equation}\label{lanuit}
\liminf_{n\to+\infty}  \frac{1}{(2n+1)^d}\log\mu(\mathcal{E}'_{n,\varepsilon}) \geq
-\inf_{u\,\in\, \left]\E_{\mu'}(f)-\frac{\varepsilon}{3},\,\E_{\mu'}(f)+\frac{\varepsilon}{3} \right[ } \,I_f(u)
\end{equation}
where
\[
I_f(u)=\inf\left\{h(\nu|\mu) : \nu\in \EuScript{M}_\Theta(\Omega), \E_{\nu}(f)=u\right\}.
\]
The right-hand side of \eqref{lanuit} is larger than $-I_f(v)$ for any value of $v$ taken in the interval, so in particular it is larger than
$-I_f(\E_{\mu'}(f))$, which is equal to $0$ because $h(\mu'|\mu)=0$ by Theorem \ref{thm:vp}. Hence we obtain
\[
\liminf_{n\to+\infty}  \frac{1}{(2n+1)^d}\log\mu(\mathcal{E}'_{n,\varepsilon}) =0
\]
which contradicts \eqref{candya}.

\appendix

\section{Appendix}

\subsection{An estimate}\label{est}

We first recall Young's inequality for convolutions \cite[p. 316]{bullen}.
Let $\ushort{u}=(u_x)_{x\in\Zd}$ and $\ushort{v}=(v_x)_{x\in\Zd}$.
Formally define their convolution $\ushort{u}* \ushort{v}$ by
\[
(\ushort{u}* \ushort{v})_x=\sum_{y\in\Zd} u_{x-y} v_y,\;x\in\Zd.
\]
If $\ushort{u}\in \ell^p(\Zd)$ and $\ushort{v}\in\ell^q(\Zd)$, where $p,q\geq 1$, then $\ushort{u}*
\ushort{v}\in\ell^r(\Zd)$ where $r\geq 1$ is such that $1+r^{-1}=p^{-1}+q^{-1}$, then we have
\[
\|\ushort{u}* \ushort{v}\|_{\ell^r(\Zd)}\leq
\|\ushort{u}\|_{\ell^p(\Zd)}\|\ushort{v}\|_{\ell^q(\Zd)}.
\]

\begin{lemma}\label{lem:deltasums}
Let $f:\Omega\to\R$ such that $\|\ushort{\gdelta}(f)\|_1:=\sum_{x\in\Zd} \delta_x(f)<+\infty$.
Then for any $\Lambda\Subset\Z^d$ we have
\[
\| \ushort{\gdelta}(S_\Lambda f)\|_2^2\leq |\Lambda | \, \| \ushort{\gdelta}(f)\|_1^2.
\]
\end{lemma}
\begin{proof}
Since $\delta_z(S_\Lambda f)\leq \sum_{x\in \Lambda}\delta_{z-x}(f)$, we apply Young's inequality
with $r=2, p=2,q=1$, $u_x=\un_{\Lambda}(x)$, and $v_x=\delta_x(f)$ to get the desired estimate.
\end{proof}

\subsection{Proof  Lemma \ref{lemma-shields}}\label{appendix:shields}

The version of this lemma in dimension $d=1$ is stated without proof in \cite{martonshields}.
Since it is not completely obvious, we give it here for any $d\geq 1$.

We fix $\varepsilon>0$ and $k\geq 0$.
The frequency of a pattern $p_k\in S^{\Lambda_k}$ in $\omega$ (see \eqref{def:empfreq}) can rewritten as
\[
\mathfrak{f}_{n,k}(\omega;p_k)=\frac{1}{(2(n-k)+1)^d}
\sum_{x\in   \Lambda_{n-k}} \un_{\{(\theta_x \omega)_{ \Lambda_k}=p_k\}}.
\]
By definition we have
\begin{align}
\nonumber
& \|\mathfrak{f}_{n,k}(\omega;\cdot)- \mathfrak{f}_{n,k}(\eta;\cdot)\|_{\tv}=\\
&\frac{1}{2(2(n-k)+1)^d}\sum_{p_k \in S^{ \Lambda_k}}
\left| \sum_{x\in   \Lambda_{n-k}} \big(\un_{\{(\theta_x \omega)_{ \Lambda_k}=p_k\}} -\un_{\{(\theta_x \eta)_{ \Lambda_k}=p_k\}}\big)\right|.
\label{kusmitea}
\end{align}
Letting
\[
\mathcal{I}_{\omega,\eta,n}=\big\{ x\in  \Lambda_{n-k}: (\theta_x \omega)_{ \Lambda_k}=(\theta_x \eta)_{ \Lambda_k}\big\}
\]
we get
\begin{align}
\nonumber
\MoveEqLeft[4] \sum_{p_k \in S^{ \Lambda_k}}\left| \sum_{x\in   \Lambda_{n-k}}
\big(\un_{\{(\theta_x \omega)_{ \Lambda_k}=p_k\}} -\un_{\{(\theta_x \eta)_{ \Lambda_k}=p_k\}}\big)\right|\\
\nonumber
&= \sum_{p_k \in S^{ \Lambda_k}}
\left| \sum_{x\in \mathcal{I}_{\omega,\eta,n}^c}
\big(\un_{\{(\theta_x \omega)_{ \Lambda_k}=p_k\}} -\un_{\{(\theta_x \eta)_{ \Lambda_k}=p_k\}}\big)\right|\\
\nonumber
& \leq \sum_{p_k \in S^{ \Lambda_k}}
\sum_{x\in \mathcal{I}_{\omega,\eta,n}^c}
\big|\un_{\{(\theta_x \omega)_{ \Lambda_k}=p_k\}} -\un_{\{(\theta_x \eta)_{ \Lambda_k}=p_k\}}\big|\\
\nonumber
& \leq \sum_{p_k \in S^{ \Lambda_k}}
\sum_{x\in \mathcal{I}_{\omega,\eta,n}^c}
\big(\un_{\{(\theta_x \omega)_{ \Lambda_k}=p_k\}} +\un_{\{(\theta_x \eta)_{ \Lambda_k}=p_k\}}\big)\\
\nonumber
& = \sum_{x\in \mathcal{I}_{\omega,\eta,n}^c}
\sum_{p_k \in S^{ \Lambda_k}}\big(\un_{\{(\theta_x \omega)_{ \Lambda_k}=p_k\}} +\un_{\{(\theta_x \eta)_{ \Lambda_k}=p_k\}}\big)\\
\nonumber
& = 2 \,\big| \mathcal{I}_{\omega,\eta,n}^c\big|.
\end{align}
Hence we obtain from \eqref{kusmitea}
\begin{equation}\label{house}
\|\mathfrak{f}_{n,k}(\omega;\cdot)- \mathfrak{f}_{n,k}(\eta;\cdot)\|_{\tv}
\leq \frac{\big| \mathcal{I}_{\omega,\eta,n}^c\big|}{(2(n-k)+1)^d}.
\end{equation}
We now look for an upper bound for $\big| \mathcal{I}_{\omega,\eta,n}^c\big|$.
If $(\theta_x \omega)_{ \Lambda_k}=p_k$ and $(\theta_x \eta)_{ \Lambda_k}\neq p_k$, then
$\omega_y\neq \eta_y$ for at least one site $y\in  \Lambda_k+x$. Such a $y$ can produce as
many as $(2k+1)^d$ sites such that $(\theta_x \omega)_{ \Lambda_k}=p_k$ and $(\theta_x \eta)_{ \Lambda_k}\neq p_k$.
Hence
\begin{align*}
\big| \mathcal{I}_{\omega,\eta,n}^c\big|
& \leq (2k+1)^d \,\big|\{x\in  \Lambda_{n-k}: \omega_x\neq \eta_x  \}\big| \\
& \leq (2k+1)^d \, \big|\{x\in  \Lambda_{n}: \omega_x\neq \eta_x  \}\big| \\
& \leq (2k+1)^d  \, \db_{ \Lambda_n}(\omega,\eta).
\end{align*}
Hence \eqref{house} yields
\[
\|\mathfrak{f}_{n,k}(\omega;\cdot)- \mathfrak{f}_{n,k}(\eta;\cdot)\|_{\tv}
\leq \frac{(2k+1)^d}{(2(n-k)+1)^d} \,\db_{ \Lambda_n}(\omega,\eta).
\]
Obviously there exists $\breve{N}>k$ such that for all $n\geq \breve{N}$ we have
\[
\left(\frac{2n+1}{2(n-k)+1}\right)^d\leq \frac{5}{4}
\]
therefore, if we take
\[
\db_{ \Lambda_n}(\omega,\eta)\leq \frac{2\varepsilon}{5(2k+1)^d}\, (2n+1)^d
\]
we finally obtain
\[
\|\mathfrak{f}_{n,k}(\omega;\cdot)- \mathfrak{f}_{n,k}(\eta;\cdot)\|_{\tv}
\leq \frac{\varepsilon}{2}
\]
for all $n\geq \breve{N}$, which concludes the proof of the lemma.

\subsection{A bound on relative entropy}

Recall that `$\log$' stands for the natural logarithm. We were not able to find a reference for a proof of the following estimate, so we prove it for the reader's convenience.

\begin{lemma}\label{lem:boundrelent}
Let $\nu$ and $\mu$ be probability measures on a finite set $A$. Then
\begin{equation}\label{eq:oeuf}
\sum_{a\,\in A} \nu(\{a\}) \left| \log\frac{\nu(\{a\})}{\mu(\{a\})}\right|\leq H(\nu|\mu)+\frac{2}{\e}
\end{equation}
where
\[
H(\nu|\mu)=\sum_{a\,\in A} \nu(\{a\}) \log\frac{\nu(\{a\})}{\mu(\{a\})}.
\]
\end{lemma}
\begin{proof}
Define
\[
A^{\scriptscriptstyle{-}}=\left\{a\in A: \log\frac{\nu(\{a\})}{\mu(\{a\})}<0\right\}.
\]
Now
\begin{align*}
\sum_{a\,\in A} \nu(\{a\}) \left| \log\frac{\nu(\{a\})}{\mu(\{a\})}\right|
&= \sum_{a\,\in A\backslash A^{\scriptscriptstyle{-}}} \nu(\{a\}) \log\frac{\nu(\{a\})}{\mu(\{a\})}
+ \sum_{a\,\in A^{\scriptscriptstyle{-}}} \nu(\{a\}) \log\frac{\mu(\{a\})}{\nu(\{a\})}\\
&= H(\nu|\mu)+ 2\sum_{a\,\in A^{\scriptscriptstyle{-}}} \nu(\{a\}) \log\frac{\mu(\{a\})}{\nu(\{a\})}.
\end{align*}
By the concavity of the logarithm function and Jensen's inequality we get
\begin{align*}
\sum_{a\,\in A^{\scriptscriptstyle{-}}} \nu(\{a\}) \log\frac{\mu(\{a\})}{\nu(\{a\})}
&= \nu(A^{\scriptscriptstyle{-}}) \sum_{a\,\in A^{\scriptscriptstyle{-}}} \frac{\nu(\{a\})}{\nu(A^{\scriptscriptstyle{-}}) } \log\frac{\mu(\{a\})}{\nu(\{a\})}\\
&\leq  \nu(A^{\scriptscriptstyle{-}}) \log\left(\frac{\mu(A^{\scriptscriptstyle{-}})}{\nu(A^{\scriptscriptstyle{-}})}\right)\\
&\leq  \nu(A^{\scriptscriptstyle{-}}) \log\left(\frac{1}{\nu(A^{\scriptscriptstyle{-}})}\right)\leq \e^{-1}
\end{align*}
where we used the elementary inequality $-x\log x\leq \e^{-1}$, $x\geq 0$.
Therefore we arrive at \eqref{eq:oeuf}.
\end{proof}

\noindent\textbf{Acknowledgements.}
We thank Pierre Collet for stimulating discussions. The authors also thank the anonymous referee for very useful comments.


\end{document}